\documentclass[reqno,11pt,a4paper]{amsart}

\usepackage[english]{babel}
 \usepackage{amsmath,amssymb}
\usepackage{graphicx}
\usepackage{color}
\usepackage{subfigure}
\usepackage{hyperref}
\usepackage{wasysym}
\usepackage{pdfsync}
\usepackage{enumitem}
\usepackage[dvipdfm,a4paper,left=30mm,right=30mm,top=35mm,bottom=35mm]{geometry}

\newcommand{\e}{\varepsilon}

\newcommand{\ti}{\tilde}
\newcommand{\al}{\alpha}

\newcommand{\ga}{\gamma}
\newcommand {\sg} {\sigma} 
\renewcommand {\th} {\theta} 
\newcommand {\tx} {\ti x}

\newcommand{\la}{\label}

\newcommand{\re}{\eqref}

\newcommand{\R}{\mathbb{R}}

\newcommand {\f}   {\frac}

\newcommand{\beq}{\begin{equation}}
\newcommand{\eeq}{\end{equation}}

\newtheorem{thm}{Theorem}[section]

\newtheorem{lem}[thm]{Lemma}
\newtheorem{rem}[thm]{Remark}
\newtheorem{cor}[thm]{Corollary} 

 \title[BGP solutions of a Boltzmann mean field game]{Balanced growth path solutions of a Boltzmann mean field game model for knowledge growth} 

\author{Martin Burger}
\address{Institute for Computational and Applied Mathematics, University of M\"unster, Einsteinstrasse 62, 48149 M\"unster, Germany}
\email{martin.burger@wwu.de} 
\author{Alexander Lorz}
\address{Sorbonne Universit\'es, UPMC Univ Paris 06, UMR 7598, Laboratoire Jacques-Louis Lions, F-75005, Paris, France; CNRS, UMR 7598, Laboratoire Jacques-Louis Lions, F-75005, Paris, France; INRIA-Paris-Rocquencourt, EPC MAMBA, Domaine de Voluceau, BP105, 78153 Le Chesnay Cedex, and CSMSE Division
King Abdullah University of Science and Technology (KAUST) Thuwal 23955-6900,
Saudi Arabia.}
\email{alexander.lorz@upmc.fr} 
\author{Marie-Therese Wolfram}
\address{University of Warwick, Coventry CV4 7AL, UK and Radon Institute for Computational and Applied Mathematics, Austrian Academy of Sciences, Altenbergerstr. 69, 4040 Linz, Austria}
\email{m.wolfram@warwick.ac.uk}

\begin{document}

\begin{abstract}
In this paper we study balanced growth path solutions of a Boltzmann mean field game model proposed by Lucas et al \cite{LM2013} to model knowledge growth in an economy.
Agents can either increase their knowledge level by exchanging ideas in learning events or by producing goods with the knowledge they already have. 
The existence of balanced growth path solutions implies exponential growth of the overall production in time. We proof existence of balanced growth path solutions if the initial distribution of individuals with respect to their knowledge level satisfies
a Pareto-tail condition. Furthermore we give first insights into the existence of such solutions if in addition to production and knowledge exchange the
knowledge level evolves by geometric Brownian motion. 
\end{abstract}

\maketitle

\section{Introduction}

\noindent Economic growth measures the inflation-adjusted increase in the market value of goods and services in an economy. A common measure is the gross-domestic product (GDP).
The GDP of the most developed countries has grown about by two percent per year since World War II. This sustained growth supports the idea of so called 
balanced growth path (BGP), which correspond to trajectories along which certain functions grow exponentially in time. Understanding what conditions initiate balanced
growth in the long-run has been an active area of research. Different mathematical models have been proposed to describe substantial growth, which can be roughly grouped into 
exogenous and endogenous growth models. Endogenous growth theory is primary based on the assumption that economic growth is related to ingenuous forces, such as human capital, internal
policies or innovation. In these models investments in ingenuous factors may lead to substantial growth. This is in contrast to exogenous growth models, which are based on the assumption that economic
prosperity is primarily determined by external rather than internal factors.\\

\noindent In this paper we study the existence of BGP solutions for an endogenous growth model proposed by Lucas and Moll, see \cite{LM2013}. It is based on the assumption that
knowledge growth in an economy is promoted by 'imitation' and 'innovation'. Imitation corresponds to learning from others, innovation to developing new ideas through
experimentation. Luttmer \cite{L2012, L2012-2} assumed that agents are 
characterized by their knowledge level and exchange knowledge in meetings. Innovation is incorporated
 via additional Brownian motion. His model serves as a starting point for the model proposed by Lucas et al \cite{LM2013}, who linked the meeting/interaction frequency of agents to an optimal choice. Hence
agents decide how much time they spend on learning and how much on producing goods. Their model corresponds to a Boltzmann mean-field game (BMFG), which we shall detail below.\\
In all these models meetings between agents are described using mathematical tools and methods from statistical mechanics, in particular kinetic theory. Kinetic theory was initially developed by Ludwig Boltzmann to analyze the statistical behavior of a system not in equilibrium \cite{boltzmann}, for example to describe the thermodynamics of dilute gases,  and has led to extensive
research on its mathematical properties (cf. \cite{cercignani1988boltzmann,diperna1989cauchy,villani2002review} and references therein). The Boltzmann equation describes 
the evolution of the probability distribution function due to microscopic interactions, for example collisions of particles or meetings between agents. Analogous approaches  have been proposed for
various applications in socio-economic sciences recently, for example price formation \cite{BCMW2013, DMT2008}, opinion exchange \cite{T2006, BS2009,DMPW2009,PT2014} or non-cooperative games \cite{DLR2014}. For a general overview on interacting multiagent systems and kinetic equations we refer to \cite{PT2013}.\\
\noindent Boltzmann mean field game (BMFG) models were recently introduced in macroeconomics, international trade and finance. While in classical mean field games the evolution
of the agent distribution is described by a Fokker-Planck equation, in BMFG models interactions between agents and their effect on the overall 
dynamics are given by a Boltzmann-type equation. Each agent determines its interaction frequency by mini- or maximizing a given cost or utility functional, resulting
in a coupling to a Hamilton-Jacobi Bellman equation. \\

\noindent Lucas et al \cite{LM2013} consider a continuum
of agents, which are characterized by their knowledge level $z \in \R^+$ and the time they devote to learning $s = s(z,t)$. Each individual has one unit of time, which
he/she can split between producing goods with the knowledge already obtained or meeting other individuals to enhance the knowledge level. Meetings are modeled by collisions
in which an individual compares its knowledge level $z \in \R^+$ with the knowledge level $z' \in \R^+$ of the other and leaves with the larger of the two, that is
\begin{align}
z = \max(z, z').
\end{align}
Let $\alpha = \alpha(s)$ denote the probability of an individual who spends an $s$-th fraction of its time on learning to  meet someone with a higher knowledge level.
Then the distribution of agents $f = f(z,t)$ with respect to their knowledge level is described by
the following Boltzmann-type equation:
\begin{align}\label{e:boltzmann}
\partial_t f(z,t) = -\alpha(s(z,t)) f(z,t) \int_z^{\infty} f(y,t)dy + f(z,t) \int_0^z \alpha(s(y,t)) f(y,t) dy.
\end{align}
The first term on the right hand side describes the loss due to interaction with a higher knowledge level, that is $y > z$. The second term the gain due to
meetings with individuals with a lower knowledge level $y < z$.  We assume that the individual production is determined by the knowledge level and the fraction of time spent on working, i.e.,
\begin{align*}
y(t) = (1-s(z,t)) z.
\end{align*}
Then the total earnings in an economy are given by
\begin{align}\label{e:production}
Y(t)  = \int_{0}^\infty(1-s(z,t)) z f(z,t)~dz,
\end{align}
that is the distribution of individuals with respect to their knowledge level times the individual productivity. Each agent maximizes its future earnings (discounted by a given temporal
discount factor $r \in \mathbb{R}^+$) by choosing the optimal fraction of time $s = s(z,t)$ spend on learning. Then this optimal fraction of learning time, is determined by the solution $s = s(z,t)$ of the optimal control problem
\begin{align*}
  V(x,t') = \max_{s \in \mathcal{S}} \bigl[\int_{t'}^T \int_0^{\infty} {e^{-r(t-t')}} (1-s(z,t)) z  \rho_x(z,t) dzdt \bigr],
\end{align*}
subject to 
\begin{align*}
\partial_t \rho_{x}(z,t) = - \alpha(s) \rho_x(z,t) \int_z^{\infty} f(y,t)\,dy + f(z,t) \int_0^z\alpha(s) \rho_x(y,t)\,dy
\end{align*}
with $\rho_x(z,t') = \delta_x$.
Here $\mathcal{S}$ denotes the set of admissible controls given by
\begin{align*}
\mathcal{S} = \lbrace s: [0,\infty) \times [0,T] \rightarrow [0,1] \rbrace.
\end{align*}

\noindent The optimal strategy can be calculated via the Lagrange functional, see  \cite{BLW2014} for details.
The optimality condition with respect to $f$ corresponds to the Hamilton-Jacobi-Bellman equation for the  value function $V = V(z,t)$:
\begin{align}
\partial_t V(z,t) - r V(z,t) + \max_{s \in \mathcal{S}}\bigl((1-s(z,t))z + \alpha(s) \int_z^{\infty}[V(y,t)-V(z,t)] f(y,t) dy\bigr) = 0.\la{e.V}
\end{align}
\noindent The situation detailed above can be summarized by the following BMFG system:
\begin{subequations}\label{e:bmfg1}
\begin{align}
&\partial_t f(z,t) = -\alpha(S(z,t)) f(z,t) \int_z^{\infty} f(y,t)dy + f(z,t) \int_0^z \alpha(S(y,t)) f(y,t) dy. \label{e:boltz}\\
&\partial_t V(z,t) - r V(z,t) = -\max_{s \in \mathcal{S}}\left[(1-s(z,t))z + \alpha(s(z,t)) \int_z^{\infty}[V(y,t)-V(z,t)] f(y,t) dy \right] \label{e:hjb}\\
&S(z,t) = \arg \max_{s \in \mathcal{S}} \left[(1-s(z,t)) z + \alpha(s(z,t)) \int_z^{\infty}[V(y,t)-V(z,t)] f(y,t) dy\right],\label{e:S0}\\
&f(z,0) = f_0(z),\\
&V(z,T) = 0.
\end{align}
\end{subequations}
For further details on the underlying modeling assumptions we refer to \cite{LM2013} and \cite{BLW2014}.\\

\noindent Local in time existence and uniqueness of solutions to \eqref{e:bmfg1} was shown by Burger et al in \cite{BLW2014}. 
 Lucas et al \cite{LM2013} 
postulated the existence of special solutions to \eqref{e:bmfg1} corresponding to exponential growth of the overall production \eqref{e:production}. 
First analytic results about the existence of BGP solutions in special situations were provided in \cite{BLW2014}. In this work we present a full analysis 
for the existence of BGP solutions and discuss conditions under which those type of solutions exist. \\

\noindent This paper is organized as follows: in Section \ref{s:bgp} we discuss the notion of balanced growth path solutions and state some analytic results from \cite{BLW2014} which
we use in the following. Existence of BGP solutions is shown in Section \ref{s:existencebgp}. We conclude by discussing the existence of BGP solutions in the case
of knowledge diffusion and presenting numerical simulations supporting our claims in Section \ref{s:knowledgediff}.

\section{Balanced Growth Paths}\label{s:bgp}

\noindent We start by introducing the notion of balanced growth path solutions for system \eqref{e:bmfg1}. 
 Assume there exists a constant $\gamma$ with $x = z e^{-\gamma t}$ such that we can define the new functions
\begin{align}\label{e:rescalbgp}
&f(z,t) =  e^{-\ga t}\phi(ze^{-\ga t}), ~~V(z,t)= e^{\ga t}v(ze^{-\ga t}) \text{ and }s(z,t)= \sg(ze^{-\ga t}).
\end{align}
 Then the Boltzmann mean field game \eqref{e:bmfg1} in $(\phi, v, \sigma) = (\phi(x), v(x), \sigma(x))$ becomes  
\begin{subequations}\label{e:bgp2}
\begin{align}
-\ga \phi(x) -\ga x \phi'(x) &= \phi(x)\int_0^x\al(\sg(y))\phi(y)\,dy - \al(\sg(x))\phi(x)\int_x^\infty\phi(y)\,dy\label{e:phi}\\
(r-\ga)v(x)+\ga x v'(x) &=  \max_{s \in \Xi}\left[(1-s)x+\al(s)\int_x^\infty[v(y)-v(x)]\phi(y)\,dy \right] \la{e:v}\\
S(x) &= \arg \max_{s \in \Xi}\left\{(1-s)x+\al(s)\int_x^\infty[v(y)-v(x)]\phi(y)\,dy \right\}\label{e:S}
\end{align}
\end{subequations}
where $\Xi = \lbrace s: \R^+ \to [0,1] \rbrace$ denotes the set of admissible controls. Then the rescaled production function \eqref{e:production} reads 
$$Y(t) =  \int_{0}^{\infty} (1-S(x)) e^{\gamma t} x e^{-\ga t} \phi(x) e^{\ga t} dx  = e^{\ga t} \int_{0}^{\infty}(1-S(x)) x \phi(x) dx,
$$
giving  exponential growth in time.\\
 A necessary prerequisite
 is the assumption that the initial cumulative distribution function of agents with respect to their knowledge level has a Pareto tail:
\begin{enumerate}[label=(A\arabic*), start=1] 
\item \label{a:pareto}  The productivity function $F(z,0) = \int_0^z f_0(y) dy$ has a Pareto tail, i.e. there exist constants $k, \theta \in \R^+$ such that
 \begin{align}
\lim_{z \to \infty} \f{1-F(z,0)}{z^{-1/\theta}} = k. \label{e:paretoF}
\end{align}
Condition \eqref{e:paretoF} in the rescaled variable $\phi$ reads as:
\beq 
\lim_{z\to \infty} \f{1-\int_0^z\phi(y)\,dy}{z^{-1/\theta}}=k. \la{e.pareto}
\eeq
\end{enumerate}
In this case the growth parameter $\gamma$ is determined by
\beq\la{e.ga0}
\ga= \th \int_0^\infty \al(\sigma(y))\phi(y)\,dx,
\eeq
see Lemma 4.2 in \cite{BLW2014}. Hence exponential growth relates to a 'fat-tailed' distribution of the initial agent density. We will see that this is not the only possibility to obtain substantial
growth. Achdou et al. \cite{ABLLM2014} postulate that knowledge diffusion may also lead to balanced growth, an idea which we shall discuss in Section \ref{s:knowledgediff}.
In \cite{BLW2014} we prove existence for the subsystem \re{e:phi}, \re{e.ga0} when the maximizer $S$ is given.

\subsection{Preliminaries}

In the following we state assumptions, notations and theoretical results from \cite{BLW2014}, which we shall use throughout this paper. 
We use capital letters to refer to the cumulative distribution functions of  $f$ and $\phi$, defined by
\begin{align*}
  F(z,t) = \int_0^z f(y,t) dy \quad \text{ and } \quad \Phi(x) = \int_0^x \phi(y) dy.
\end{align*}
To simplify the notation in the following we use the additional function
\begin{align*}
  B(x) := \int_x^{\infty} (v(y)-v(x)) \phi(y) dy.
\end{align*}
\noindent Since $\alpha = \alpha(s)$ corresponds to the interaction probability to engage in a meeting (which initiates learning) we shall also refer to it as the term learning function 
throughout this paper. Furthermore  we need the following assumptions:
\begin{enumerate}[label=(A\arabic*), start=2]
\item \label{a:f0} Let $f_0 \in L^{\infty}(\R^+)$ be a probability density with $\int_{0}^\infty f_0(y)dy = 1$ and $f_0(z) \geq 0$ for all $z$.
\item \label{a:alpha} Let $\alpha: [0,1]  \to \R^+$, $\alpha \in C^{\infty}([0,1]),~\alpha(0) = 0,~\alpha'(0) = \infty$, $\alpha''<0$ and $\alpha$ monotone.
\end{enumerate}

\noindent In case of a given function $\alpha$ we were able to prove the existence of a  solution $\phi$ which has a Pareto tail. We recall this result from \cite{BLW2014}:
\begin{thm}
Let assumption \ref{a:alpha} hold and $\sigma\in C^1([0,\infty))$ denote a given function, which satisfies
\begin{align*}
\sigma(z) = 1 \text{ for } z \in [0, z_0],~\sigma'(z) \leq 0.
\end{align*}
Then there exists a $\gamma \in \R^+$ and a  solution $\phi \in L^1([0,\infty))$ to equation \eqref{e:phi}, which has a Pareto tail.
\end{thm}

\noindent In the proof of Lemma 3.6 in \cite{BLW2014} we state Lipschitz properties of the  maximisers of the right hand side of \eqref{e:S}. The lemma reads as:
\begin{lem} Let assumption \ref{a:alpha} be satisfied, $z >0$, $B \in \R$ and $S=S(B)$ be the optimal control for a given $B$ when maximizing $(1-s)x + \al(s) B$ over $s \in [0,1]$. If
\begin{align*}
\lim_{B\to0}\al''(S(B))B^3<0,
\end{align*}
then the maps
$B \to S(B)$, $B\to \al(S(B))$ and $B \to \al(S(B))B$ are Lipschitz.
\end{lem}
\begin{rem}\la{re.max}
In the proof of the preceding lemma we identify the three possible cases when maximizing $(1-s)x + \al(s) B$ over $s \in [0,1]$:\\
Case 1: If $B \le 0$, then $S=0$.\\
Case 2: If $B \al'(1)\ge x$, then $S=1$.\\
Case 3: If $0<B \al'(1)< x$, then $\al'(S)=\f x B$.\\
Since we assume $\al'(0)=\infty$, we obtain that the maximiser $S$ is continuous in $B$ and $x$ for $x>0$.
\end{rem}

\subsection{Properties of Solutions}\label{s:propbgp}

Let us mention some properties related to the existence or nonexistence of balanced growth path (BGP) solutions in the following. First of all, we see that a certain behavior of the initial value at infinity is always needed for exponential growth from the results in \cite{BLW2014}, where it was shown that the support of a solution $f$ is always a subset of the support of the initial value. The relation to the specific tail of the initial value 
was further worked out in the case of a constant learning function $\alpha = \alpha_0$. Under this condition system \eqref{e:bgp2} decouples and we can calculate the BGP solutions explicitly (cf.  \cite{ABLLM2014,BLW2014}): 
\begin{thm} \label{theoremalphaconst}
Let assumption \ref{a:pareto} be satisfied and $\al = \alpha_0$. Then there exists a unique $(\Phi, v, 0)$ and a scaling constant $\gamma$ to \eqref{e:bgp2} given by
\begin{align*}
\gamma =  \alpha_0 \theta, \quad\Phi(x) =\f{1}{1+kx^{-1/\theta}},
\end{align*}
solving \eqref{e:bgp2}, \eqref{e:paretoF} and \eqref{e.ga0}.
\end{thm}

A general issue of the BGP-system \eqref{e:bgp2}, which has strong impact on the analysis and computation is the existence of degenerate solutions with $\gamma = 0$, $v \equiv x/r$ and $S \equiv 0$. In this case we see that necessarily $\Phi(x)=1$ for $x > 0$, which means that $\phi$ is a Dirac-$\delta$ concentrated at $x=0$. In order to exclude such in the analysis of BGP solutions and their numerical computation it is essential to construct solutions $\Phi$ that satisfy a specific Pareto tail condition with some $k > 0$, such that the Dirac-$\delta$ that would lead to 
$k=0$ is not admissible. Hence, we shall in particular consider a  transformation of variables from
$x$ to $\ti x = x^{-1/\theta}$ and solve the correspondingly transformed equation with an initial condition at $\ti x =0$. This also corresponds to the result in the case of constant $\alpha$ given in Theorem \ref{theoremalphaconst}: we can only expect uniqueness of the solution under the additional asymptotic condition \eqref{e:paretoF} at infinity, in particular for each $k$ we expect a different solution.

\section{Existence of BGP solutions}\label{s:existencebgp}

\noindent In this section we prove existence of solutions related to balanced growth. We start by  showing several analytic results for the rescaled Boltzmann equation and the Hamilton-Jacobi-Bellman
equation, which are necessary ingredients in the existence proof for the full system. Let us recall the full BGP system:
\begin{subequations}\la{e.bgp_full}
\begin{align}
&\ga  x \Phi'(x)  = [1-\Phi(x)]\int_0^x\al(S(y))\Phi'(y)\,dy \label{e.Phi}\\
&(r-\ga)v(x)+\ga x v'(x) = \max_{s\in\mathcal{S}}\left\{(1-s)x+\al(s)\int_x^\infty[v(y)-v(x)]\Phi'(y)\,dy \right\} \la{e.v2}\\
&S(x) = \arg \max_{s \in \mathcal{S}}\left\{(1-s)x+\al(s)\int_x^\infty[v(y)-v(x)]\Phi'(y)\,dy \right\}\la{e.S}\\
&\ga= \th \int_0^\infty \al(S(y))\Phi'(y)\,dy\la{e.ga}
\end{align}
\end{subequations}
where $\mathcal{S} = \lbrace s: \R^+ \to [0,1] \rbrace$ denotes the set of admissible controls. 

\noindent Based on their derivation and interpretation we shall call equation \eqref{e.Phi} the BGP-Boltzmann equation and \eqref{e.v2} the BGP-Hamilton-Jacobi-Bellman equation (or briefly BGP-HJB equation).

\subsection{Existence and Uniqueness for the BGP-HJB Equation}

First we discuss the existence and uniqueness of solutions $v = v(x)$ to \eqref{e.v2} given
a function $\Phi = \Phi(x)$ and a positive constant $\gamma \in \R^+$.

\noindent We start with a regularization of the BGP-HJB equation. Let us consider
\begin{align}
(r-\ga)v_\e(x)+\ga(x+\e) v'_\e(x)  &= \max_{s\in\Xi}\left\{(1-s)x+\al(s)\int_x^\infty[v_\e(y)-v_\e(x)]\Phi'(y)\,dy \right\} \la{e.vr}
\end{align}
for $\e \ge 0$ on $\R_+$ and denote the  maximiser by $S_\e$. 
\begin{lem}\la{le.v'>=0}
A solution $v_\e$ of equation \re{e.vr} with boundary condition $v'_\e(0) \ge 0$ is non-decreasing. The maximiser $S_\e$ is non-increasing.
\end{lem}
\begin{proof}

Assume there exists an $x$ such that $v'_\e(x) <0$. Then there is 
 a $x_0$ and an $\eta$ small such that $v'_\e(x_0) =0$  and $v'_\e(x) <0$ for $x \in [x_0,x_0 + \eta]$. Considering the left-hand side of equation  \re{e.vr} it follows that 
$$(r-\ga)v_\e(x_0)+\ga x_0 v'_\e(x_0) >  (r-\ga)v_\e(x_0+\e)+\ga (x_0+\eta) v'_\e(x_0+\eta),$$
whereas both terms on the right-hand side  of equation  \re{e.vr} are increasing in $x$. This is a contradiction.
So for $\e \ge 0$ we have $v'_\e \ge 0$ everywhere. 

\noindent Next we show that $S_\e$ is non-increasing. Using  $B'_\e(x)= -v'_\e(x) (1-\Phi(x))$ and Remark \ref{re.max}, we have one of the two cases:
\begin{enumerate}
\item $B_\e(0) >0$: For small $x$, we have $B_\e(x) \al'(1)\ge x$ therefore $S = 1$. As $x$ increases, there is a unique point $x_0(S_\e)$ such that $B_\e(x_0(S_\e)) \al'(1)= x_0(S_\e)$. For $x$ larger than $x_0(S_\e)$ and $B_\e(x) >0$, we have $\al'(s)=\f x B$, so $S_\e$ is strictly decreasing. If there is a point $x_1$ such that $B_\e(x_1)=0$, then $B_\e(x)=0$ for $x \ge x_1$ and therefore the maximiser $S_\e$ is equal to $0$ on $[x_1,\infty)$.\\
\item $B_\e(0) =0$: In this case $B_\e$ is equal to $0$ on $[0,\infty)$ and therefore the maximiser $S_\e$ is equal to $0$ on $[x_1,\infty)$.
\end{enumerate}
In both cases, the maximiser $S_\e$ is non-increasing.
\end{proof}

\noindent From now on we shall use the definition of $x_0(S)$ (already introduced in the proof of the preceeding lemma) to be the unique point where $B(x_0(S)) \al'(1)= x_0(S)$.
\begin{rem}\la{re.xS}
For a non-increasing maximiser $S$  with $S(0)=1$ and $S(x) \to 0$ for $x\to \infty$, the following holds: for $x<x_0(S)$ we have $S(x)=1$ and for  $x>x_0(S)$ we have $S(x)<1$. So $x_0(S)$ is the point between the intervals where $S$ is equal to $1$ and $S$ less than $1$.
\end{rem}

\begin{lem}\la{le.v_const_on_interval}
A solution $v_\e$ of equation \re{e.vr} satisfying the boundary condition $v'_\e(0) = 0$ is constant on the interval $[0,x_0(S_\e)]$.
\end{lem}
\begin{proof}
On the interval $[0,x_0(S_\e)]$,  the maximiser $S_\e$ equals $1$, hence the right-hand side of equation \re{e.vr} equals $\alpha(1) B(x)$. Since $B'(x) = - v'(x) (1-\Phi(x))$, $\alpha(1) B(x)$ is non-increasing. We know from Lemma \ref{le.v'>=0}  that $v_\e' \geq 0$ and find on the left-hand side of equation \re{e.vr} that
$$(r-\ga)v_\e(0) = (r-\ga)v_\e(0) +\ga (0+\e) v'_\e(0) \le  (r-\ga)v_\e(x)+\ga (x+\e) v'_\e(x).$$ 
Thus, both sides of equation \re{e.vr} are non-increasing and non-decreasing, hence constant. Therefore $v_\e$ is constant.
\end{proof}

\begin{thm}\la{le.v_exist}
Let $\Phi \in L^\infty(\R^+)$ be a function with $\Phi' \in L^1_+(\R^+)$, $\Phi(0)=0$, $\lim_{x\to\infty}\Phi(x)=1$ and parameter $\gamma \ge 0$ be given. Then equation \re{e.v2} with the boundary condition $v'(0)=0$ has a solution $(v,S)$ with $v \in L^\infty_{\f 1 {1+x}}(\R^+)$ and $S \in L^\infty(\R^+)$. Moreover, $v$ satisfies 
\begin{equation} \label{e:vprimebounds}
	0 \leq v'(x) \leq \frac{1}r , \qquad \text{ for all } x \in \R^+.
\end{equation}
\end{thm}
\begin{proof}
The proof is divided into several steps. First we deduce an equivalent formation of the original formulation to construct a fixed point operator. Then we proof existence using the Leray Schauder theorem.\\
\noindent For $x<x_0(S)$ we differentiate the equation for $v$ once  and obtain
$$r v' + \ga x v''+ \al(S) v' (1-\Phi(x)) = 0= 1-S.$$
If we differentiate \eqref{e.v2} for  $x>x_0(S)$ we obtain
$$r v' + \ga x v'' = 1-S - S'x + \al'(S)S'B + \al(S) B' = 1-S - S'x +  \f x B S'B + \al(S) B'. $$
So to the left and to the right of $x_0(S)$, we have
\begin{align}
r v' +\ga x v''+ \al(S) v' (1-\Phi(x)) = 1-S. \label{e:vprime}
\end{align}
As an integral of a $L^1$-function the term $1-\Phi(x)$  is continuous and according to Remark \ref{re.max}, $S$ is continuous. So all coefficients are continuous. Therefore the ODE \re{e:vprime} holds everywhere in $[0,\infty)$. This means finding a solution $(v,S)$ to equations \re{e.v2} and \re{e.S} is equivalent to finding a solution $(v,S)$ to equations \re{e:vprime} and \re{e.S} and motivates the following approach:
We solve the ODE \re{e:vprime} with initial condition $v'(0)=0$ such that also equation  \re{e.S} is satisfied.

\noindent First we define the fixed point operator for the
 regularized equation \re{e.vr} with a small parameter $\e > 0$. Let 
$$M:=\{ w :  w \ge 0 \text{ continuous with } 0 \le w' \le \f 1 r \}$$ 
and  $P:M\rightarrow M$ be the operator defined in the following way:  Given $w$ we define 
$$\ti S :=\arg \max_{s\in\Xi}\left\{(1-s)x+\al(s)\int_x^\infty[w(y)-w(x)]\Phi'(y)\,dy \right\}.$$ 
Using this maximiser $\ti S$ we solve the ODE 
\beq r v' + \ga (x+\e) v''+ \al(\ti S) v' (1-\Phi(x)) = 1-\ti S,
\la{e.vtiS}
\eeq
for $v' $ with $v'(0)=0$. Note that we have dropped the index $\e$ to increase readability.
Integrating with $v(0) = w(0)$ gives $v=:P(w)$.
\noindent We need to show that the operator $P$ is self-mapping, compact and continuous with respect to the norm $\|v\|_{L^\infty_{\f 1 {1+x^2}}}=\|\f v {1+x^2}\|_\infty$.
\subsubsection*{Self-mapping} 
By maximum principle, we have $0\le v' \le \f{1}{r}$. So the operator $P$ is a self-mapping. 

\subsubsection*{Compactness}
Taking a sequence $w_k$ bounded in $L^\infty_{\f 1 {1+x^2}}$-norm, we obtain that $\f {v_k}{1+x^2}$ as well as
\begin{align*}
\left(\f {v_k} {1+x^2} \right)' = \f{v'_k}{1+x^2} - \f{2 v_k x}{(1+x^2)^2}
\end{align*}
is bounded in the $L^\infty$-norm. Since  $\f {v_k}{1+x^2} x $ remains bounded, we have a momentum bound which implies compactness.

\subsubsection*{Continuity}
Let us take a sequence $w_k$ converging to $w$ in $L^\infty_{\f 1 {1+x^2}}$-norm and define 
\begin{align*}
v_k := M(w_k) \text{ and } v := M(w).
\end{align*}
For every subsequence $v_k$ by compactness there is a further subsequence $v_k$ converging to a $v_\infty$. By taking a further subsequence, we can ensure that $v'_k$ converges weakly-$\star$ to $v'_\infty$. Moreover, $v_k$ converges pointwise and by dominated convergence also $B_k$ converges pointwise. Therefore the maximiser $\ti S_k$ converges pointwise and we have convergence of $\ti S_k$ and $\al(\ti S_k)$ in $L^1_{loc}$.
By passing to the limit $k\to \infty$ in the weak formulation of the equation 
$$r v'_k + \ga (x+\e)v''_k+ \al(\ti S_k) v' _k(1-\Phi(x)) = 1-\ti S_k$$
and using uniqueness, we have $v=v_\infty$ and therefore $v_k$ converging to $v$ in $L^\infty_{\f 1 {1+x^2}}$-norm. 

\subsubsection*{Bounded subset of $M$}
To apply Leray-Schauder fixed point theorem, we define 
$\ti M := \{w : w = \lambda Pw \text{ for } 0 \le \lambda \le 1 \}.$
Since we require $w(0)=(Pw)(0)$ in the definition of the operator $P$,  for $w \in \ti M$ we have either  $w(0)=0$ or $\lambda=1$. If $w(0)=0$, with the bound $0 \le w' \le  \lambda \f 1 r \le \f 1 r $, it follows that all those $w$ are uniformly bounded in $L^\infty_{\f 1 {1+x^2}}$. 
For $\lambda=1$, we have $w(0) = \int_0^\infty[w(y)-w(0)]\Phi'(y)\,dy = \int_0^\infty w'(y)(1-\Phi(y)) dy  \le  \int_0^\infty \f 1 r(1-\Phi(y))\,dy$.
Since $\Phi$ has a Pareto-tail with $\theta <1$ and $0 \le w' \le \f 1 r$, the set $\ti M$ is bounded.
By the Leray-Schauder theorem we obtain a fixed point. 

\subsubsection*{Limit $\e \to 0$}
Including the $\e$-dependence, we have a solution to the equation
\begin{multline}
(r-\ga)v_\e(x)+\ga (x+\e) v'_\e(x)  = \max_{s\in\Xi}\left\{(1-s)x+\al(s)\int_x^\infty[v_\e(y)-v_\e(x)]\Phi'(y)\,dy \right\} \\
=(1-S_\e)x+\al(S_\e)\int_x^\infty[v_\e(y)-v_\e(x)]\Phi'(y)\,dy
\la{e.v_e}
\end{multline}
We still obtain $0\le v'_\e \le \f{1}{r}$, and therefore as in the paragraph on compactness, we can extract a subsequence $v_\e$ converging to $v$ in $L^\infty_{loc}$. Moreover, we also obtain $S_\e \to S$ and $\al(S_\e) \to \al(S)$ in $L^1_{loc}$. With Remark \ref{re.max} this is enough to pass to the limit.\\
We extract a further subsequence such that $x_0(S_\e)$ converges to $x_0(S)$. Assume there exists an $x_0(S)$ that is equal to $0$. Then $B(0)$ is equal to $0$ which implies that $B=0$ everywhere. Hence $S=0$. It follows that $v(x)=x/r$ and therefore $\phi$ is a Dirac delta at $0$, which is a contradiction. Therefore $x_0(S)$ has to be larger than $0$.\\
By Lemma \ref{le.v_const_on_interval}, the solution $v_\e$ is constant on $[0,x_0(S_\e)]$, so the limit $v$ is constant on $[0,x_0(S)]$ and therefore the solution $v$ satisfies the boundary condition $v'(0)=0$. The identity \eqref{e:vprimebounds} follows from the construction.
\end{proof}

\noindent Under the appropriate boundary condition on $v'(0)$ we can further prove uniqueness for the BGP-HJB equation:

\begin{lem}Under the assumptions of Lemma \ref{le.v_exist}, the solution $(v,S)$ of \re{e.vr} with $v'(0)=0$ is unique.
\end{lem} 
\begin{proof}
Let us assume there are two solutions $v, w$ for a given $\Phi$ with Pareto-tail satisfying
$$(r-\ga)v(x)+\ga x v'(x)  =  \max_{s\in\Xi}\left\{(1-s)x+\al(s)\int_x^\infty[v(y)-v(x)]\Phi'(y)\,dy \right\}  $$
$$(r-\ga)w(x)+\ga x w'(x)  =  \max_{s\in\Xi}\left\{(1-s)x+\al(s)\int_x^\infty[w(y)-w(x)]\Phi'(y)\,dy \right\}.$$
Let us define $B_v(x) := \int_x^\infty[v(y)-v(x)]\Phi'(y)\,dy$,  $B_w(x) := \int_x^\infty[w(y)-w(x)]\Phi'(y)\,dy$ and write 
 the corresponding equations for $v'$ and $w'$:
 $$r v' + \ga xv''+ \al(S_v) v' (1-\Phi(x)) = 1-S_v,$$
 $$r w' + \ga xw''+ \al(S_w) w' (1-\Phi(x)) = 1-S_w.$$
 Taking the difference of these equations, we obtain
 \beq
 r (v'-w') + \ga x (v'-w')' +   (1-\Phi(x)) (\al(S_v)v'-\al(S_w)w') = S_w - S_v. \la{e.v'-w'}
 \eeq
 We consider the following cases:
\begin{enumerate}
\item If $v(0)=w(0)$ then uniqueness follows from Picard-Lindel\"of due to the Lipschitz-properties of the maximiser $S$,
\item Assume that $v(0)\not=w(0)$ and wlog that $v(0) >w(0)$.\\
Then the inequality $B_v(0) >  B_w(0)$ holds, therefore we have 
\begin{align*}
S_v(x)= 1 \ge S_w(x) \text{ on the interval } [0, \al'(1)B_v(0)] \supset [0, \al'(1)B_w(0)].
\end{align*}
If $B_v(x) >  B_w(x)$ for all $x$, then $S_v(x) \ge S_w(x)$ since 
$\al'(S_v)=\f x {B_v}$ and $\al'(S_w)=\f x {B_w}$ hold on the interval $[\al'(1)B_v(0),\infty)$. 
Since we have $\lim_{x\to\infty} v'(x) = \lim_{x\to\infty} w'(x) = \f 1 r$  by maximum principle applied to equation \re{e.v'-w'}, we obtain $v'\le w'$.  
This inequality leads to 
\begin{align*}
B'_w = -w' (1-\Phi)\le -v' (1-\Phi) = B'_v
\end{align*}
and therefore $\lim_{x\to \infty} B_w(x) < \lim_{x\to \infty} B_v(x)$, which is a contradiction.\\
If there exists an $x$ such that $B_v(x) =  B_w(x)$, let us call the minimal $x$ with this property $x_0$. As before on the interval $[0,x_0]$ we have $S_v(x) \ge S_w(x)$. Therefore by maximum principle this gives $v'\le w'$ and therefore the contraction $B_w(x_0) <  B_v(x_0)$.
\end{enumerate}
\end{proof}

\noindent Finally we derive a bound for $v$ related to its (asymptotically) linear growth:
\begin{lem}
Under the assumptions of Lemma \ref{le.v_exist}, the solution $(v,S)$ of \re{e.vr} with boundary condition $v'(0)=0$ satisfies
\begin{equation}\label{e:vbounds}
	v(x) \left\{ \begin{array}{lll} \geq \frac{x}r  & \text{ for all } x \in \R_+, &  \text{if } \gamma < r,  \\ \leq \frac{x}r  & \text{ for all } x \in \R_+, &  \text{if } \gamma > r .\end{array} \right.
\end{equation}
\end{lem}
\begin{proof}
Since the maximum at the right-hand side of \re{e.vr} can be estimated from below by the value at $S=0$, we have
$$ (r-\gamma) v(x) + \gamma v'(x) x \geq x. $$
Due to the upper bound on $v'$ this implies
$$ (r-\gamma) \left( v(x) - \frac{x}r \right) \geq 0 $$
and hence the assertion.
\end{proof}

\subsection{Existence and Uniqueness of Solutions to the BGP-Boltzmann equation with given Pareto-tail}

We shall develop a more refined strategy than in \cite{BLW2014} and use the following variable transformation and notations:
\begin{align*}
\ti x := x^{-1/\th} \text{ and } \Phi(x) =: 1-    \frac{\gamma}\theta K (\ti x) \ti x,
\end{align*}
where $\th$ and $k$ denote the Pareto indizes in  \re{e.Phi_tail}. We will determine the appropriate $\gamma$ from $\th$ and $K$ as 
\begin{equation}
	\gamma = \frac{\theta}{\lim_{\ti x \rightarrow \infty} K(\ti x) \ti x},
\end{equation}
in order to guarantee that $\Phi(0)=0$. On the other hand we have the limit 
$$ \frac{\gamma}\theta K(0)=\lim_{\ti x \rightarrow 0} \frac{\gamma}\theta K(\ti x) = \lim_{x \rightarrow \infty} \frac{1-\Phi(x)}{x^{-\frac{1}{\theta}}} = k, $$
i.e., we can use the initial value $K(0)=\frac{\theta k}\gamma$. Since $\gamma$ is unknown a-priori but determined as the limit $\ti x \rightarrow \infty$, we instead use an arbitrary initial value $\ti k$ and subsequently determine $k=\frac{\gamma}\theta \ti k$. Moreover, since $K$ is normalized by $\gamma$ and $\theta$, equation \eqref{e.Phi} reads as
\beq
\ti x K'= -K \int_0^{\ti x} \al(\ti S)(K\xi)'\,d\xi.\la{e.K1}
\eeq
Hence we look for a solution with the constraint 
\beq
\int_0^\infty \al(\ti S)(K\xi)'\,d\xi =1. \la{e.K=1}
\eeq

\noindent Adding $K$ on both sides we can write the equation for $K$ in the alternative form
\beq
(\ti x K(\ti x))'= K(\ti x)\left( 1- \int_0^{\ti x} \al(\ti S)(K(\xi)\xi)'\,d\xi \right).\la{e.K1a}
\eeq

\noindent We recall that $\ti S$ is the maximiser $S$ in the new variable $\ti x$, in particular $\ti S$ is non-decreasing and we define $\tx_0(\ti S)$ as the point such that below $\tx_0(\ti S)$ the maximiser $\ti S$ is less than $1$ and above $\tx_0(\ti S)$  the maximiser $\ti S$ equals $1$.

\begin{lem} Let $\ti S$ be non-decreasing, then
equation \re{e.K1} with $K(0)=\ti k > 0$ has a unique continuous solution with values in $[0,\ti k]$, satisfying in addition 
\begin{equation} \label{e.Kconstraints}
	K'(\ti x) \leq 0, \qquad (K(\xi)\xi)' \geq 0, \qquad \int_0^{\ti x} \al(\ti S(\xi))(K(\xi)\xi)'\,d\xi \leq 1, 
\end{equation}
for all $\xi \geq 0$. Additionally, if $\ti S$ is different from $0$, the solution satisfies the constraint \re{e.K=1}.
\end{lem}
\begin{proof}
For sufficiently small $\delta > 0$ we can formulate \re{e.K1} on the interval $[0,\delta]$ as a fixed point equation for $L=K'$ in the form
$$L(\ti x) =  -\left( \ti k + \int_0^{\ti x} L(\xi)~d\xi \right) \frac{1}{\ti x} \int_0^{\ti x} \al(\ti S)
\left(L(\xi) \xi + \ti k + \int_0^\xi L(\eta)~d\eta \right)\,d\xi,$$
with the obvious continuation $L(0) = -\al(1)\ti k^2$. For $\delta$ sufficiently small it is straight-forward to verify a contraction property and boundedness in the supremum norm, from which one obtains the existence and uniqueness of a fixed point in the space of continuous functions. Subsequently we can reconstruct $K(\ti x) = \ti k + \int_0^{\ti x} L(\xi)~d\xi .$ Moreover, for $\delta$ sufficiently small it is straight-forward to verify \re{e.Kconstraints} and nonnegativity of $K$. The upper bound $K(\ti x) \leq \ti k$ follows with the nonpositivity of $L$.

\noindent Given the values $K(\delta)$ and $L(\delta)$ we can solve for $\tx > \delta$
$$L(\ti x) =  -( K(\delta) + \int_\delta^{\ti x} L(\xi)~d\xi) \frac{1}{\ti x} \int_\delta^{\ti x} \al(\ti S)
\left(L(\xi) \xi + K(\delta) + \int_\delta^\xi L(\eta)~d\eta\right)\,d\xi,$$
as a well-posed Volterra integral equation uniquely by standard arguments. 

\noindent It remains to verify the bounds on $K$ and \eqref{e.Kconstraints} for arbitrary $\ti x>\delta$. First of all we have $\log K(\delta)$ finite and see from \re{e.K1} that $(\log K)'$ is bounded for every $\ti x$, which implies that $\log K$ is finite, i.e., $K$ is nonnegative.
Now consider the equation in the form \re{e.K1a} and assume there exists a value $\ti x_0$ such that $(K(\xi)\xi)' \geq 0$ for $\xi \leq \ti x_0$ and $(K(\xi)\xi)'< 0$ in the interval $(\ti x_0,\ti x_0+\delta)$ for some $\delta > 0$. Then we see from \re{e.K1a} that $$ \int_0^{\ti x_0} \al(\ti S(\xi)) (K(\xi)\xi)'\,d\xi  = 1.$$ Due to the sign change in $(K(\xi) \xi)'$ we find
$  \int_0^{\ti x} \al(\ti S(\xi)) (K(\xi)\xi)'\,d\xi  < 1$ for $\ti x \in (\ti x_0,\ti x_0+\delta).$
Inserting this relation into \re{e.K1a} yields a contradiction to the negativity of $(K(\ti x)\ti x)'$. With the nonnegativity of $K$, $\alpha$ and $(K(\xi)\xi)'$ we immediately obtain $K'\leq 0$ and hence $K(\xi) \leq K(0)=\ti k$. The integral inequality in \re{e.Kconstraints} follows now immediately from  \re{e.K1a}.

\noindent If $\ti S$ is not identically zero, then by its monotonicity there exist $\epsilon >0$ and an interval $(\ti x_1,\infty)$ such that $\al(\ti S) \geq \epsilon$ on this interval. Hence, $(K(\xi)\xi)'$ is integrable on an unbounded interval, which implies that it tends to zero for $\xi \rightarrow \infty$. 
Assuming $ \int_0^\infty \al(\ti S)(K\xi)'\,d\xi < 1$ leads to 
$(K(\tx) \ti x)' \ge \f{c}{\tx} $ for $\tx$ large and some constant $c> 0$. This contradicts the integrability of $(K(\tx) \ti x)'$ however. Therefore constraint \re{e.K=1} is satisfied.
\end{proof}

\begin{lem} \label{le.gammatheta}
Let $S(0)=1$, $x_0(S) > 0$, and $S$ non-increasing. Then  the limit
\begin{equation}
	 \gamma = \frac{\theta}{\lim_{\ti x \to \infty} (\ti x K(\ti x))}
\end{equation}
is positive and finite. In particular
\begin{equation}
	K(\ti x)  \geq \frac{1}{\frac{1}{\ti k} + \int_{ 0}^{\ti x} \alpha(\xi) ~d\xi } \geq
	\frac{1}{\frac{1}{\ti k} + \alpha(1) \ti x} \qquad \text{ for all } \ti x \geq 0
\end{equation}
and
\begin{equation}
	K(\ti x) \ti x \leq \ti k \ti x_0(\ti S)+ \frac{1}{\alpha(1)}\qquad \text{ for all } \ti x \geq 
	\ti x_0(\ti S) 
\end{equation}
hold, which implies
\begin{equation}
	\alpha(1) \geq \frac{\gamma}{\theta} \geq \frac{1}{\ti k \ti x_0(\ti S) + \frac{1}{\alpha(1)}} \la{e.ga_th}.
\end{equation}
\end{lem} 
\begin{proof}
For $\ti x \geq \ti x_0(S)$ we have 
$$
1= \int_0^\infty \al(\ti S)  (K(\xi)\xi)'\,d\xi \geq \int_{\ti x_0(\ti S)}^{\ti x} \al(\ti S) (K(\xi)\xi)'\,d\xi = \al(1) ( K(\ti x) \ti x- K(\ti x_0(\ti S))\ti x_0(\ti S)) , $$
which immediately yields (together with the fact that $K$ is non-increasing and $K(0)=\ti k$) that
$$    K(\ti x) \ti x \leq  \ti k \ti x_0(\ti S) + \frac{1}{\al(1)}. $$
This gives a bound for the limit $\ti x \to \infty$ respectively the lower bound for $\frac{\gamma}\theta.$

\noindent We use the alternative formulation \eqref{e.K1a} and the estimates from Lemma \ref{le.gammatheta} in the following.
Since $\alpha(\ti S(\ti x))$ is nondecreasing we have 
$$ \int_0^{\ti x} \alpha(\ti S(\xi)) (K(\xi) \xi)' ~d\xi \leq \alpha (\ti S(\ti x))
\int_0^{\ti x}  (K(\xi) \xi)' ~d\xi  = \alpha (\ti S(\ti x)) K(\ti x) \ti x. $$
Thus, 
\beq 
\ti x K'(\ti x) \geq  - \alpha (\ti S(\ti x)) K(\ti x)^2 \ti x,\la{e.K'}
\eeq 
and therefore 
$$ - \frac{K'(\ti x)}{ K(\ti x)^2} \leq \alpha (\ti S(\ti x)).$$
This implies
$$ \frac{1}{K(\ti x)} - \frac{1}{\ti k} \leq \int_{ 0}^{\ti x} \alpha(\xi) ~d\xi $$
and consequently gives the lower bound on $K$. Furthermore we obtain a uniform lower bound for the limit
$$ \lim_{\ti x \rightarrow \infty} K(\ti x) \ti x  \geq
\lim_{\ti x \rightarrow \infty}	\frac{\ti x}{\frac{1}{\ti k} + \alpha(1) \ti x} = \frac{1}{\alpha(1)}, $$
which implies in particular $\frac{\gamma}\theta \leq  \alpha(1)$. 
\end{proof}

\noindent From the results above we can immediately deduce the following statement:
\begin{cor}
For $\ti k > 0$,  equation \re{e.Phi} has a unique solution $(\Phi,\ga)$ with 
\beq  \lim_{x\to \infty} \f{1-\Phi(x)}{ x^{-1/\th}}= \frac{\gamma \ti k}{\theta} = k\la{e.Phi_tail}.\eeq
\end{cor}

Note that we obtain existence and uniqueness of a solution for any $\ti k$, but not necessarily for each $k$, since due to the implicit dependence of $\gamma$ on $\ti k$ the map
$\ti k \mapsto k$ is not necessarily surjective.

\subsection{Analysis of the coupled BGP-System}

In the following we show the existence of a non-trivial balanced growth path, i.e. existence for the system \eqref{e.Phi}, \eqref{e.v2}, \eqref{e.S}, \eqref{e.ga}. Our idea is to construct a fixed-point map by first solving the \eqref{e.Phi}, \eqref{e.ga} given $v$ and $S$ and subsequently \eqref{e.v2}, \eqref{e.S} for $(v,S)$ given $\phi$. The previous sections establish the well-definedness of all steps, but we see that in order to obtain reasonable bounds we need to set up the fixed point map on set that bounds $x_0(S)$ away from zero. Hence we first need some estimates for $x_0(S)$, which means an estimate on the set of $x$ such that $B(x) \alpha'(1) \geq x$. Therefore we need to obtain a lower bound for $B(x)$ for small $x$, which we perform in the following:

\begin{lem} \label{le:B0bound}
Let $\gamma < r$ and $\Phi$ satisfy 
$$ 1- \Phi(x) \geq \frac{\gamma}{\theta(\alpha(1)+ \frac{1}{\ti k}x^{\frac{1}\theta})}.$$
Furthermore let $v$ be the unique solution of equation \re{e.v2} with $v'(0)=0$. 
Then the following inequality holds 
\begin{equation}
	B(0) \geq   \frac{\gamma(r-\gamma)}{\theta r(r-\gamma+\alpha(1))}  I(\ti k),~\text{ with }~ I(\ti k) = \int_0^\infty  \frac{1}{(\alpha(1)+ \frac{1}{\ti k}y^{\frac{1}\theta})}~dy.
\end{equation}
\end{lem} 
\begin{proof}
Since $v'(0)=0$ and $S=1$ is a maximizer at $x=0$ we have $(r-\gamma) v(0) = \alpha (1)B(0)$ and
$$ \frac{r-\gamma+\alpha(1)}{\alpha(1)} v(0) = - \int_0^\infty v(y) (1-\Phi)'(y)~dy.
$$ 
Now the monotonicity of $(1-\Phi)$ and $v(y) \geq \frac{y}r$ imply
\begin{eqnarray*}
 B(0) &\geq& \frac{r-\gamma}{r-\gamma+\alpha(1)} \int_0^\infty -\frac{y}r   (1-\Phi)'(y)~dy \\
&=& \frac{r-\gamma}{r(r-\gamma+\alpha(1))} \int_0^\infty  (1-\Phi(y))~dy \\
&\geq& \frac{\gamma(r-\gamma)}{\theta r(r-\gamma+\alpha(1))} \int_0^\infty  \frac{1}{(\alpha(1)+ \frac{1}{\ti k}y^{\frac{1}\theta})}~dy,
\end{eqnarray*}
where we have used integration by parts in the second and the lower bound on $1-\Phi$ in the last step.
\end{proof}

\noindent This means that under conditions verified by a solution $\Phi$ of the BGP-Boltzmann equation, the value of $B(0)$ is bounded away from zero.
This gives the following lemma:

\begin{lem}\la{le.x_1}
Let in addition to the assumptions  of Lemma \ref{le:B0bound} the condition \re{e.ga_th}
be satisfied for $\gamma$. Then there exists a point $x_1>0$ independent of the specific $\gamma$ and $\Phi$ such that $x_0(S) > x_1$.
\end{lem} 
\begin{proof}
Using the bound  \re{e.ga_th} for $x \in [0,x_0(S)]$
 we have the estimate 
\begin{multline} B(x_0(S)) \alpha'(1) = B(0) \alpha'(1) \ge \frac{\gamma(r-\gamma)\al'(1)}{\theta r(r-\gamma+\alpha(1))}  I(\ti k) \\ \geq \frac{(r-\theta \alpha(1))\al'(1)}{(\ti k x_0(S)^{-\frac{1}\theta}+\frac{1}{\alpha(1)})r(r-\frac{\theta}{\ti k x_0(S)^{-\frac{1}\theta}+\frac{1}{\alpha(1)}}+\alpha(1))}  I(\ti k) =: F(x_0(S)).
\end{multline}

\noindent The function $F$ is  continuous  with $F(0) > 0$. Now let $x_1$ be the minimal value such that
$F(x_1) = x_1$. Then we have 
$B(x_1) \alpha'(1) \ge F(x_1) = x_1 $
and therefore  $x_0(S) > x_1$.
\end{proof}


\noindent With this lower bound for $x_1$ we can construct a self-mapping with $x_0(S)$ uniformly bounded away from zero in the fixed-point argument:

\begin{thm}
Let $r> \th \al(1)$ and $\ti k>0$, then the system \re{e.bgp_full} has a non-trivial solution satisfying the Pareto-tail condition \re{e.Phi_tail} with $k=\frac{\gamma}{\theta} \ti k$.
\end{thm}
\begin{proof}
Let $M_k$ be the set of all functions $K \in C([0,\infty))$ such that 
\begin{enumerate}[label=(\roman*)]
\item $K(0)=\ti k$, 
\item $	K(\ti x)  \geq 
	\frac{1}{\frac{1}{\ti k} + \alpha(1) \ti x} ,\text{ for all } \ti x \geq 0,
$
\item $(\ti x K( \ti x))'$ exists and is non-negative,
\item $\lim_{\ti x \to \infty} (\ti x K(\ti x))$ exists and is positive.
\end{enumerate}
Then we define the operator $Q$ in the following way:
For $K \in M_k$, we define as above $$\ga := \frac{\theta}{\lim_{\ti x \to \infty} (\ti x K(\ti x))}, \qquad
\Phi(x) := 1- \f \ga \th  \ti x K(\ti x).$$
Given $\Phi$ we solve equation \re{e.v2} to obtain $v$ and $S$. Next we solve equation \re{e.K1} for $\hat K$.
\noindent We need to show that $Q$ is a self-mapping, compact and continuous w.r.t. the $L^\infty$-norm.
\subsubsection*{Self-mapping}
By Lemma \ref{le.x_1} we have $S(0)=1$, therefore the properties of $\hat K$ from Lemma \ref{le.gammatheta} give that the operator $Q$ is a map from $M_k$ to $M_k$. 
\subsubsection*{Compactness}
We take a sequence $K_n$ in $M_k$ and define $\hat K_n := Q(K_n)$. Then $\hat K_n$ is bounded above by $\ti k$ and because of inequality \re{e.K'} $\hat K'_n$ is bounded below by $\ti k^2 \al(1)$. According to Lemma \ref{le.x_1} there is a $\ti x_1$ such that  for all $n$ $\ti S_n(\ti x) =1$ for $\ti x \ge \ti x_1$. 
From inequality \re{e.K'} we obtain 
$$\hat K_n(\ti x) \le \f{1}{\al(1)(\ti x - \ti x_1)}\qquad \text{ for } \ti x \ge \ti x_1 +1$$
and  a moment bound for $\hat K_n$ follows immediately.
Therefore there is a subsequence $\hat K_n $ converging to a $\hat K_\infty \in M_k$ in $L^\infty$-norm.
 \subsubsection*{Continuity}
We take a sequence $K_n$ in $M_k$ converging to $K$  and define $\hat K_n := Q(K_n)$, $\hat K := Q(K)$.
 As in the compactness part of Theorem \ref{le.v_exist}, we can extract a subsequence $v_n$ converging in
$L^\infty_{\f 1 {1+x^2}}$-norm. In the same way we can extract a further subsequence $B_n$ converging to $B_\infty$ in  $L^\infty_{\f 1 {1+x^2}}$-norm.
$B_\infty$  allows us to define $S_\infty$ and we have pointwise convergence of $S_n$ to $S_\infty$. Since the maximiser is bounded by $1$, we have $L^1_{loc}$-convergence of $S_n$ to $S_\infty$ and $\al(S_n)$ to $\al(S_\infty)$.
Using all these results we can pass to the limit in the equation
$$(r-\ga_n)v_n(x)+\ga_n x v'_n(x)  =  (1-S_n)x+\al(S_n)\int_x^\infty v'_n(y)(1-\Phi_n(y))\,dy .$$
By uniqueness of the solution we obtain convergence of $S_n$ to $S$ in $L^1_{loc}$ and $v_n$ to $v$ in $L^\infty_{loc}$.
Next we pass to the limit in the equation for $\hat K_n$. Using the same arguments as in the compactness part, there is a converging subsequence $\hat K_n$ with $\hat K'_n$ weakly  converging. Therefore we can pass to the limit in the term $\int_0^{\ti x} \al(\ti S_n)(\hat K_n(\xi)\xi)'\,d\xi$ and also in the equation
$$\ti x \hat K'_n(\ti x)= -\hat K_n(\ti x)\int_0^{\ti x} \al(\ti S_n)(\hat K_n(\xi)\xi)'\,d\xi.$$
Again by uniqueness we have $\hat K_n \to  \hat K$.
 \subsubsection*{Existence}
With the above conditions, the Schauder fixed point theory implies the existence of a solution, which is non-trivial since $x_0(S)>0$.
\end{proof}

\section{Knowledge diffusion leads to balanced growth}\label{s:knowledgediff}

Achdou et al. \cite{ABLLM2014} postulate that knowledge diffusion leads to balanced growth, i.e. they presume that individual productivity also fluctuates in the absence of
meetings. Similar statements were shown in related works by Alvarez et al., Lucas and Staley in \cite{ABL2008, L2009, S2011} in the case of Boltzmann type models for knowledge growth. These models are closely related
to the BMFG model \eqref{e:bmfg1} - they correspond to \eqref{e:boltzmann} with a given constant interaction probability $\alpha = \alpha_0$. In this case system \eqref{e:bmfg1} decouples and  \eqref{e:boltzmann} can be written in terms of the cumulative
distribution function $F = F(y,t)$:
\begin{align}\label{e:F}
\partial_t F(y,t) = - \alpha_0 F(y,t)(1-F(y,t)).
\end{align}
Equation \eqref{e:F} can be studied in the case of small diffusion, which models innovation by small fluctuations in the individual productivity, see \cite{ABL2008, L2009,S2011} for more details. Define $G(y,t) := 1-F(y,t)$, then \eqref{e:F} corresponds
to the Fisher-KPP equation:
\begin{subequations}\label{e:kpp}
\begin{align}
&\partial G(y,t) - \nu \partial_{yy} G(y,t) = \alpha_0 G(y,t)(1-G(y,t))\\
&\lim_{y \rightarrow -\infty} G(y,t) = 1,~\lim_{y\rightarrow \infty} G(y,t) = 0, ~ G(y,0) = 1-F(y,0),
\end{align} 
\end{subequations}
where $\nu \in \R^+$ denotes the diffusivity. It is well known that \eqref{e:kpp} admits traveling wave solutions of the form
\begin{align}
G(y,t) = \Phi(y-\gamma t),
\end{align}
with a minimal wave speed $\gamma = 2 \sqrt{\nu \alpha_0}$.\\
Note that these traveling waves are closely related to the BGP solutions introduced in Section \ref{s:bgp}. Consider the spatial variable $z = e^{y}$, then
the rescaling \eqref{e:rescalbgp} in the BGP variables reads as
\begin{align*}
f(z,t) = e^{-\ga t} \phi(z e^{-\ga t}) = e^{-\ga t} \phi(e^{y-\ga t}).
\end{align*}
Hence the scaling constant $\gamma$ corresponds to the wave speed in the logarithmic variables.\\

\noindent We pursue this idea in the following and assume that in addition to production and the learning events, the knowledge of each agent evolves by a geometric Brownian motion
$$ dZ_t = \sqrt{2 \nu} Z_t dW_t, $$
where $W_t$ is a Wiener process, independent between the agents.
The corresponding version of the Boltzmann mean field game system \eqref{e:bmfg1} with diffusion
thus becomes
\begin{subequations}\label{e:diffmfg}
\begin{align}
\begin{split}
\partial_t f(z,t) - \nu \partial_{zz}(z^2 f(z,t)) &= f(z,t) \int_0^z \alpha(S(y,t)) f(y,t) dy \\
&- \alpha(S(z,t)) f(z,t) \int_z^{\infty} f(y,t) dy,\\
\end{split}\\
\begin{split}
\partial_t V(z,t) +& \nu z^2 \partial_{zz} V(z,t) - r V(z,t) = {}\\
&-\max_{s \in \mathcal{S}}\left[(1-s) z + \alpha(s) \int_z^{\infty}[V(y,t)-V(z,t)]f(y,t) dy\right]. 
\end{split}
\end{align}
\end{subequations}
Achdou et al. \cite{ABLLM2014} postulated the existence of balanced growth path solutions to system \eqref{e:diffmfg} with a rescaling parameter $\gamma$ given by
\begin{align}\label{e:diffgamma}
\gamma = 2 \sqrt{\nu \int_0^{\infty} \alpha(S(y)) \phi(y) dy }.
\end{align}
Assuming the existence of the scaling parameter $\gamma$ we rewrite system in the known BGP variables $(\phi, \sigma, v)$ (defined by \eqref{e:rescalbgp}):
\begin{subequations}\label{e:diffbgp}
\begin{align}
-\gamma \phi(x) - \gamma x \phi'(x) - \nu (x^2 \phi(x))'' &= \phi(x) \int_0^x \alpha(\sigma(y)) \phi(y) dy - \alpha(\sigma(x)) \phi(x) \int_x^{\infty} \phi(y) dy \label{e:difffp}\\
(r-\gamma) v(x) + \gamma x v'(x) + \nu x^2 v''(x) &= \max_{\sigma \in \Sigma} \left[(1-\sigma)x + \alpha(\sigma)\int_x^{\infty}[v(y)-v(x)] \phi(y) dy \right] \label{e:diffhjb}.
\end{align}
\end{subequations}
Note that the diffusion does not exclude the existence of the degenerate solutions  $\gamma = 0$, $v(x)=\frac{x}{r}$, $S(x) = 0$ and $\phi(x) = \delta(x)$. This degeneracy has
to be considered in the design of the numerical solver which we shall detail below.

\subsection{Numerical simulations}

\noindent We illustrate the behavior of the BMFG system \eqref{e:diffmfg} and the corresponding BGP system \eqref{e:diffbgp} in the case of diffusion.
We solve both systems using iterative schemes, that is by solving consecutively the Boltzmann equation and the HJB equation  and then updating the respective variables until convergence. We shall detail the steps for both solvers in the following.\\

\noindent \textbf{The time dependent solver}\\
\noindent We discretize the Boltzmann equation \eqref{e:difffp} and HJB equation \eqref{e:diffhjb} using a semi-implicit in time  and
a finite difference discretization of the diffusion operator $\partial_{zz}(z^2 f(z,t))$ in space. We consider system \eqref{e:diffmfg} on a bounded domain $\mathcal{I} = [0,\bar{z}]$
with initial and terminal conditions
\begin{align*}
f(z,0) = f_0(z) \text{ with } \int_0^{\bar{z}} f_0(y) dy = 1 \text{ and } V(z,T) = 0.
\end{align*}
We assume that the total number of agents is conserved in time, therefore
\begin{align*}
\partial_z (z^2 f)(0,t) = 0 \text{ and } \partial_z (z^2 f)(\bar{z},t) = 0  \text{ for all } t\geq 0.
\end{align*}
Furthermore we set homogeneous Neumann boundary conditions for $V$. \\

\noindent Let $\tau$ denote the time step and $h$ the size of the spatial intervals. We shall use superindizes to refer to the
time step and subscripts for the spatial position, e.g $f^k_i$ denotes the solution $f$ at time $t^k = k \tau$ and position $x_i = ih$.\\

\noindent The solver is based on the following iterative procedure:  
\begin{enumerate}
\item Given $f_0$ and $S^{k}$ solve
\begin{align*}
\frac{1}{\tau}(f_i^{k+1}-f_i^{k}) &+ \frac{2\nu}{h}(z_{i+1} f^{k+1}_{i+1} - z_i f^{k+1}_i) \\
&+ \frac{\nu}{h^2} (z_{i+1}^2 f^{k+1}_{i+1} - (z_{i+1}^2+z_i^2) f^{k+1}_i + z_i^2 f^{k+1}_{i+1}) = g_1(f^{k}, S^{k}),\
\end{align*}
for every time $t^k = k \tau, ~ k > 1$, using a trapezoidal rule to approximate the integrals in $g_1$.
\item Update the maximizer $S^k$ defined by the right hand side of \eqref{e:diffmfg} using the evolution of $f^k$ and $V^k$.
\item Given the evolution of the density $f^k$ and the maximizer $S^k$ solve the HJB equation
\begin{align*}
\frac{1}{\tau} (V_i^{k+1} - V^{k}) - \frac{\nu}{h^2} z_i^2 (V^{k}_{i+1} - 2V^k_i + V^k_{i-1}) - r V_i^{k} = g_2(S^{k+1}, f^{k+1}, V^{k+1}),
\end{align*}
backward in time using a trapezoidal rule to approximate $g_2$.
\item Go to step (1) until convergence.
\end{enumerate}
 
\vspace*{1em}
\noindent \textit{Time dependent simulations:\\}
\noindent The simulations were performed on the interval $\mathcal{I} = (0,20)$ divided into $1000$ elements of size $h = 0.002 $. The time steps were
set to $\tau = 0.05$, the final time $T=100$. We choose the following simulation parameters:
\begin{align*}
\alpha_0 =0.075 , n = 0.3  \text{ and } r =0.05. 
\end{align*}
The initial distribution of agents is given by
\begin{align*}
f_0(z) = \frac{1}{\sqrt{2\pi}} e^{-\frac{(x-5)^2}{2}}.
\end{align*}
The behavior for $\nu = 0.005$ is illustrated in Figure \ref{f:bmfg1}. Even though the initial distribution of the agents does not have a Pareto
tail the diffusion initiates long-term growth. \\

\begin{figure}
\subfigure[Evolution of the agent density.]{\includegraphics[width=0.45\textwidth]{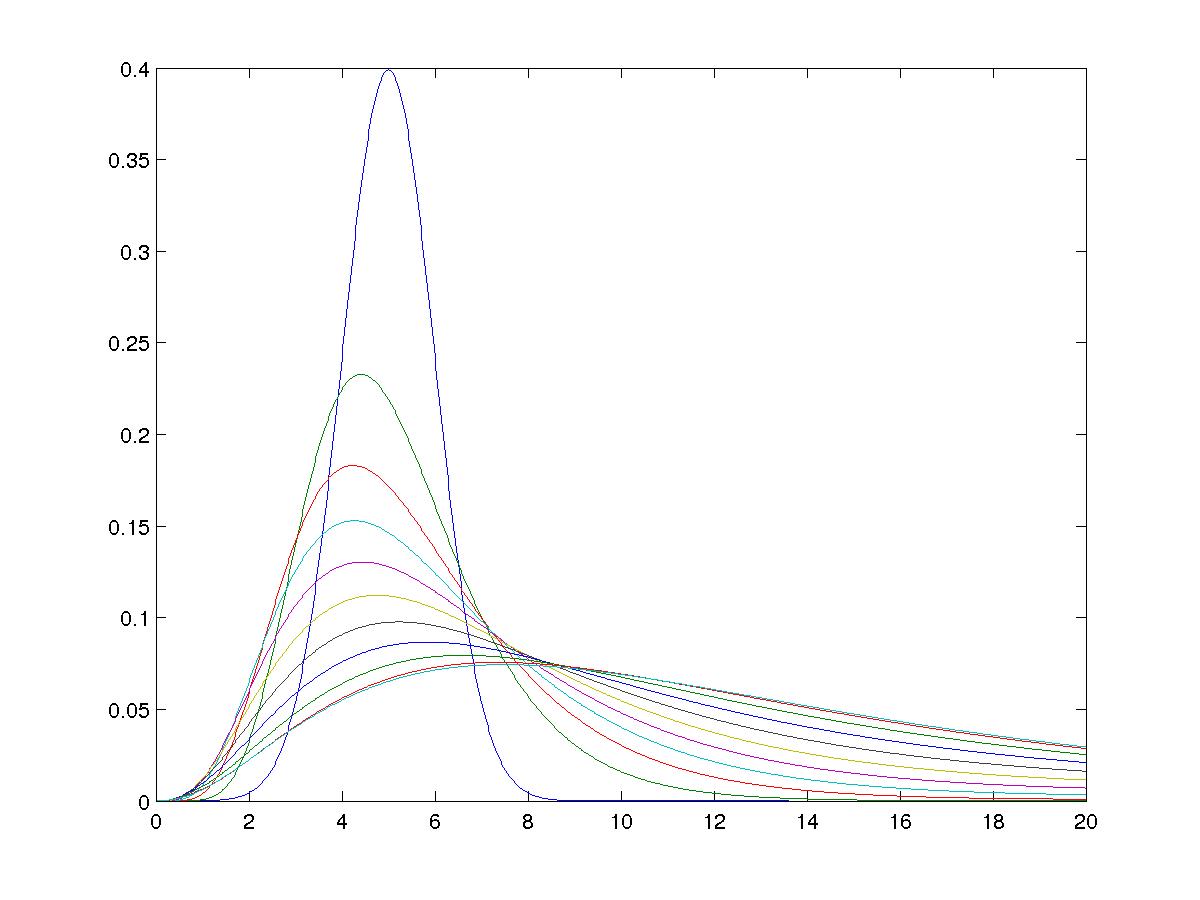}}
\subfigure[Evolution of the production function $Y$.]{\includegraphics[width=0.45\textwidth]{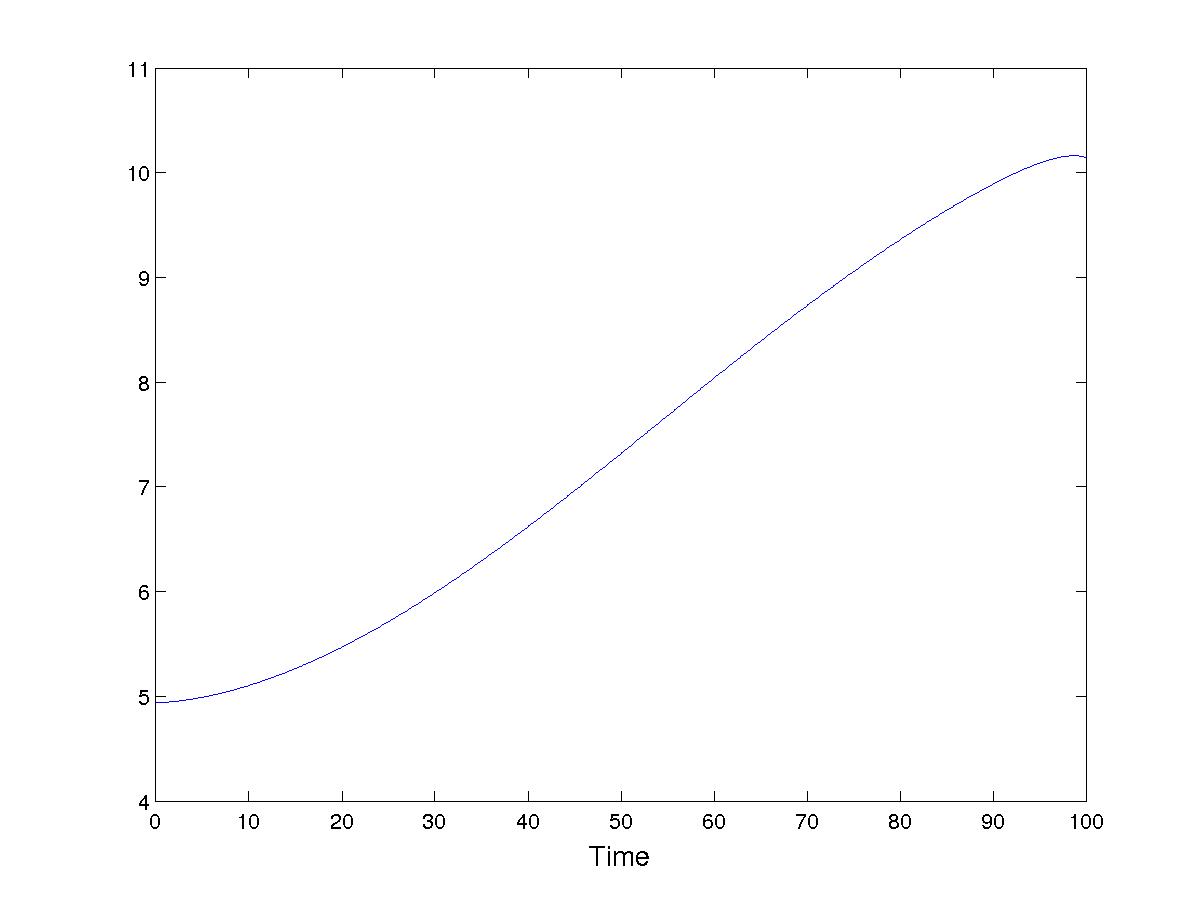}}
\caption{Solution of the time dependent solver converging to a non-trivial BGP}\label{f:bmfg1}
\end{figure}

\noindent Note that diffusion does not automatically initiate substantial growth. We discussed the existence of degenerate BGP solutions, which correspond to
the formation of a Delta Dirac in the agent distribution in Section \ref{s:propbgp} already. These degenerate solutions also cause considerable problems in the
numerical simulations as we shall illustrate in the next 
example. If we increase the diffusivity to  $\nu = 0.125$ the solver converges towards the degenerate solution, see Figure \ref{f:bmfg2}. In this case the 
overall production decreases, hence we do not obtain exponential growth.

\begin{figure}
\subfigure[Evolution of the agent density.]{\includegraphics[width=0.45\textwidth]{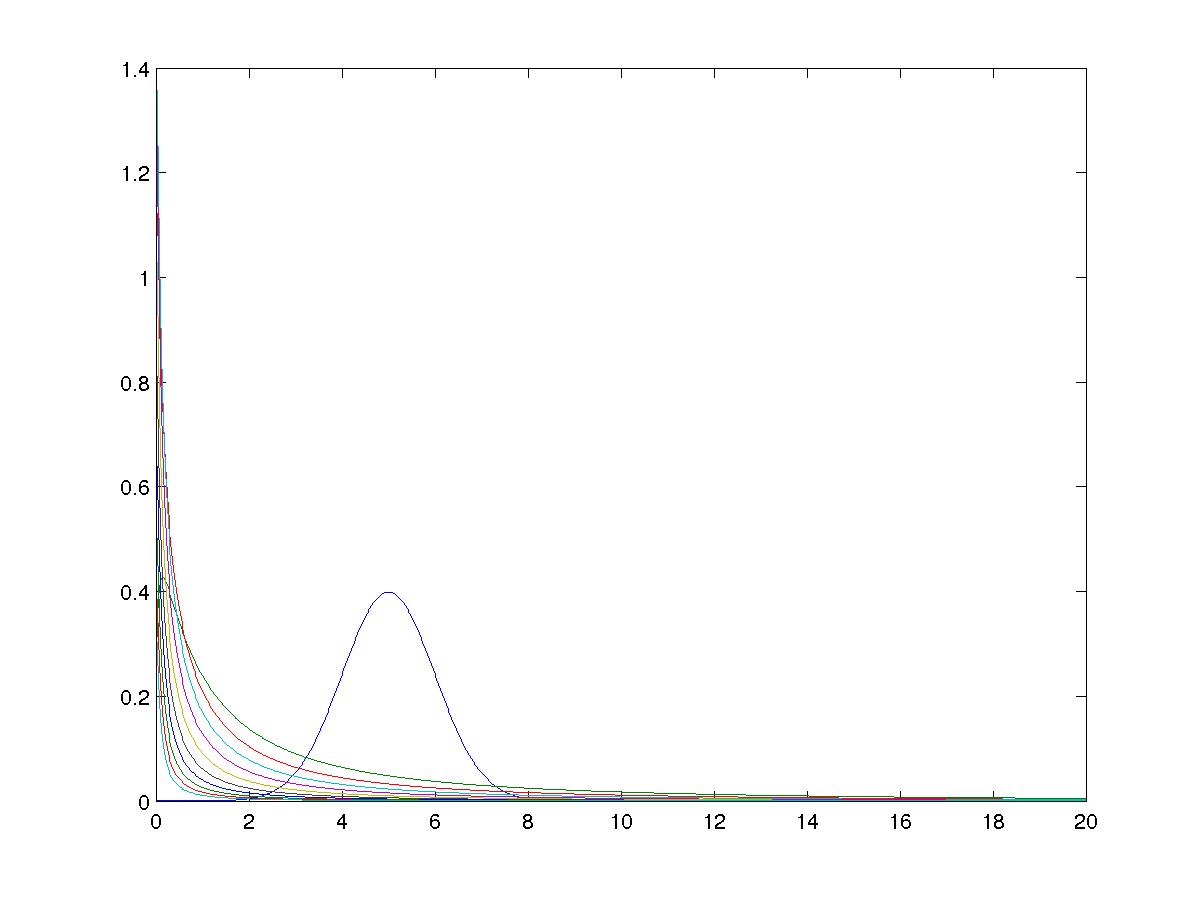}}
\subfigure[Evolution of the production function $Y$.]{\includegraphics[width=0.45\textwidth]{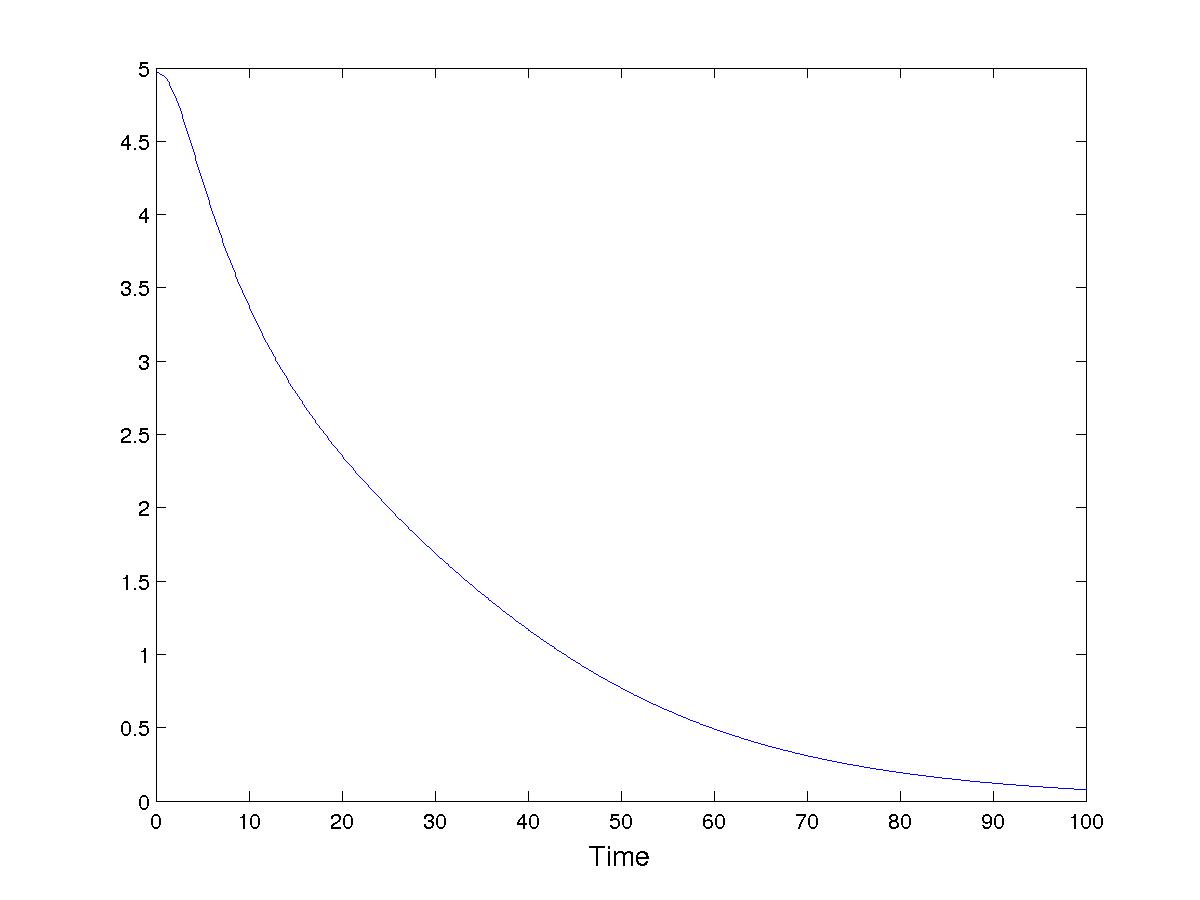}}
\caption{Solution of the time dependent solver converging to the trivial BGP}\label{f:bmfg2}
\end{figure}

\vspace*{1em}
\noindent \textbf{The BGP solver}\\
\noindent The BGP solver is based on an iterative procedure and a finite difference discretization. We use superscripts to denote the iteration number, while the subscripts refer to the spatial position.
System \eqref{e:diffbgp} is supplemented with with
boundary conditions 
\begin{align*}
\partial_x (x^2 \phi)(0) = \partial_x(x^2\phi)(\bar{x}) = 0 \text{ and } v'(0) = v'(\bar{x}) =  0 .
\end{align*}
Note that the no-flux boundary conditions for $\phi$ is automatically satisfied at $x = 0$. To exclude the existence of degenerate BGP solutions we 
set 
\begin{align*}
\phi_0= 0.
\end{align*}
\vspace*{1em}
\noindent The iterative solver is given by
\begin{enumerate}
\item Given $v^n, \sigma^n$ and $\gamma^n$ solve
\begin{align*}
-\gamma^n \phi^{n+1}_i &- \frac{\gamma^n}{h} x_i (\phi^{n+1}_{i+1}-\phi^{n+1}_i) - \frac{2 \nu}{h} (x_{i+1} \phi^{n+1}_{i+1} - x_i \phi^{n+1}_i) \\
&- \frac{\nu}{h^2} (x_{i+1}^2 \phi^{n+1}_{i+1} - (x_{i+1}^2 + x_i^2) \phi_i^{n+1} + x_i^2 \phi^{n+1}_{i-1}) = q_1(\phi^n, \sigma^n)
\end{align*}
subject to the constraint that $(\phi^{n+1}_1 + \phi^{n+1}_2 + \ldots \frac{1}{2}\phi^{n+1}_N) h = 1$ (which corresponds to the discretization of the constraint $\int_0^{\bar{z}} \phi(y) dy = 1$ 
using the trapezoidal rule and $\phi^{n+1}_0 = 0$). Also the integrals in $q_1$ are evaluated using the trapezoidal rule.
\item Given $\phi^{n+1},~\gamma^n$ and $\sigma^n$ solve 
\begin{align*}
(r-\gamma^n) v^{n+1}_i + \frac{\gamma^n}{h} x_i (v_i^{n+1}-v_{i-1}^{n+1}) -\frac{\nu x_i^2}{h^2}( v^{n+1}_{i+1} &- 2 v_i^{n+1} +  v_{i-1}^{n+1}) \\
&= -q_2(\phi_{n+1}, v^{n}, \sigma^{n})
\end{align*}
using the trapezoidal rule to approximate $q_2$.
\item Compute the maximum $\sigma^{n+1}$ and update the growth parameter $\gamma^{n+1}$ via
\begin{align*}
\gamma^{n+1} = 2 \bigl(\nu \int_{\mathcal{I}} \alpha(\sigma^{n+1}(y)) \phi^{n+1}(y) dy\bigr)^{\frac{1}{2}}.
\end{align*}
\item Go to (1) until convergence.
\end{enumerate}
\vspace*{1em}
\noindent \textit{BGP simulations for different diffusivities $\nu$:\\}
\noindent The simulations were performed on the interval $\mathcal{I} = (0,20)$ divided into $1000 $ elements of size $h = 0.02 $. We use over-relaxation to update the parameters $\phi,~v$ and $\sigma$ in each iteration using a damping parameter $\omega = 0.75$. The simulation parameters are set to:
\begin{align*}
\alpha_0 = 0.005 ,~~ n = 0.5 \text{ and } r = 0.1 . 
\end{align*}
The behavior for different values of $\nu$ is illustrated in Figure \ref{f:bgpdiff}. As expected larger diffusivities lead to stronger exponential growth. Furthermore we observe that in the
case of large diffusion the point $x_0(S)$, the point where the fraction of time devoted to learning starts to decrease, is decreasing (see Figure \ref{f:bgpdiff} (b)). \\

\begin{figure}
\subfigure[Agent distribution for different values of $\nu$.]{\includegraphics[width=0.45\textwidth]{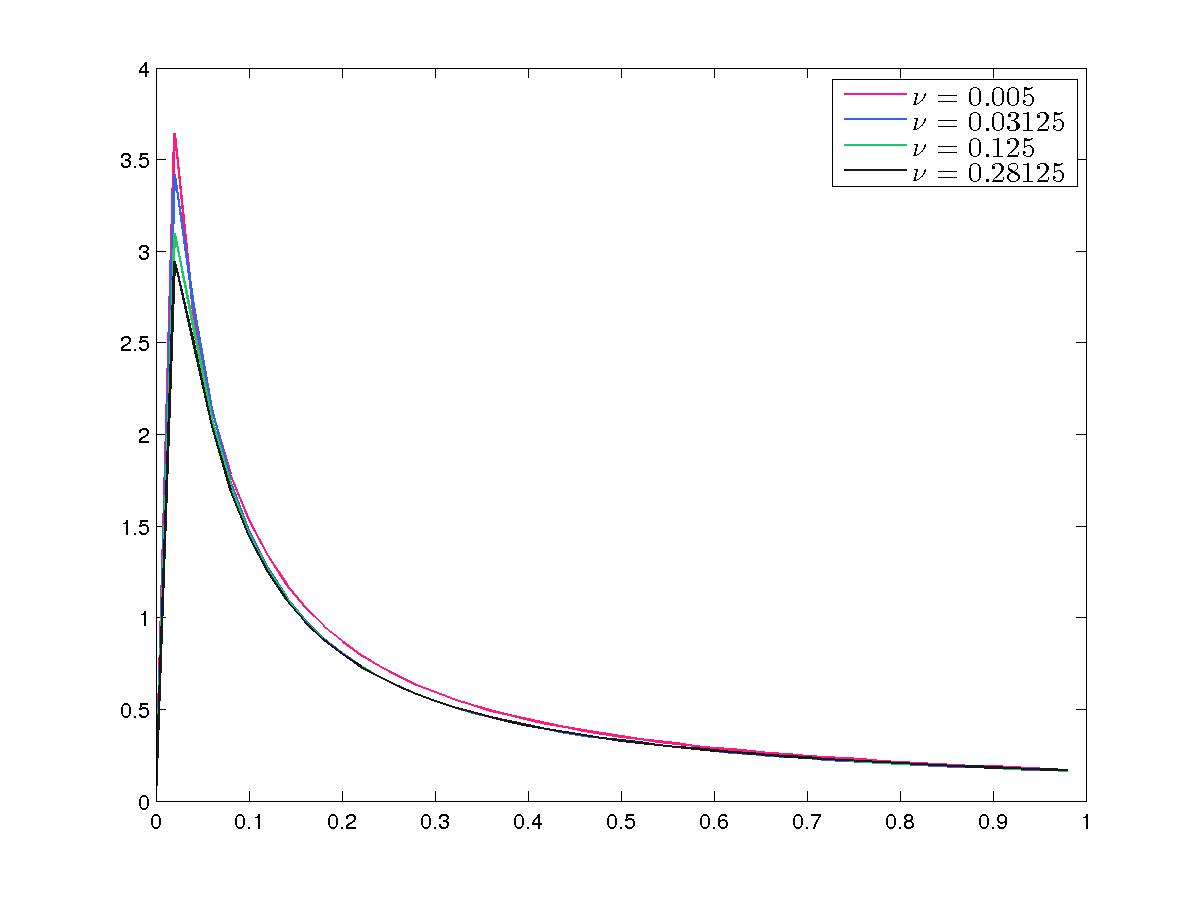}}\hspace*{0.2cm}
\subfigure[Fraction of time $\sigma$ devoted to learning for different values of $\nu$.]{\includegraphics[width=0.45\textwidth]{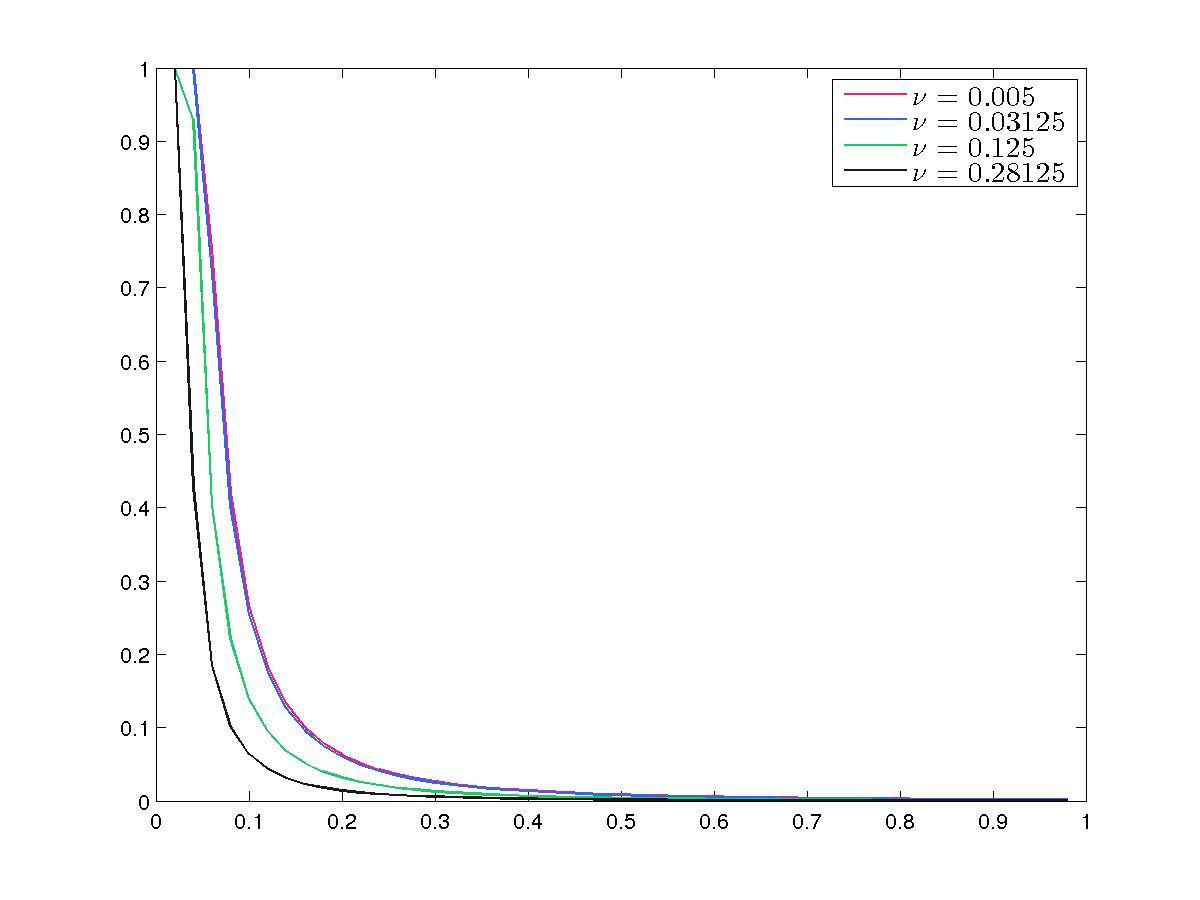}}\\
\subfigure[Production function $Y$.]{\includegraphics[width=0.4\textwidth]{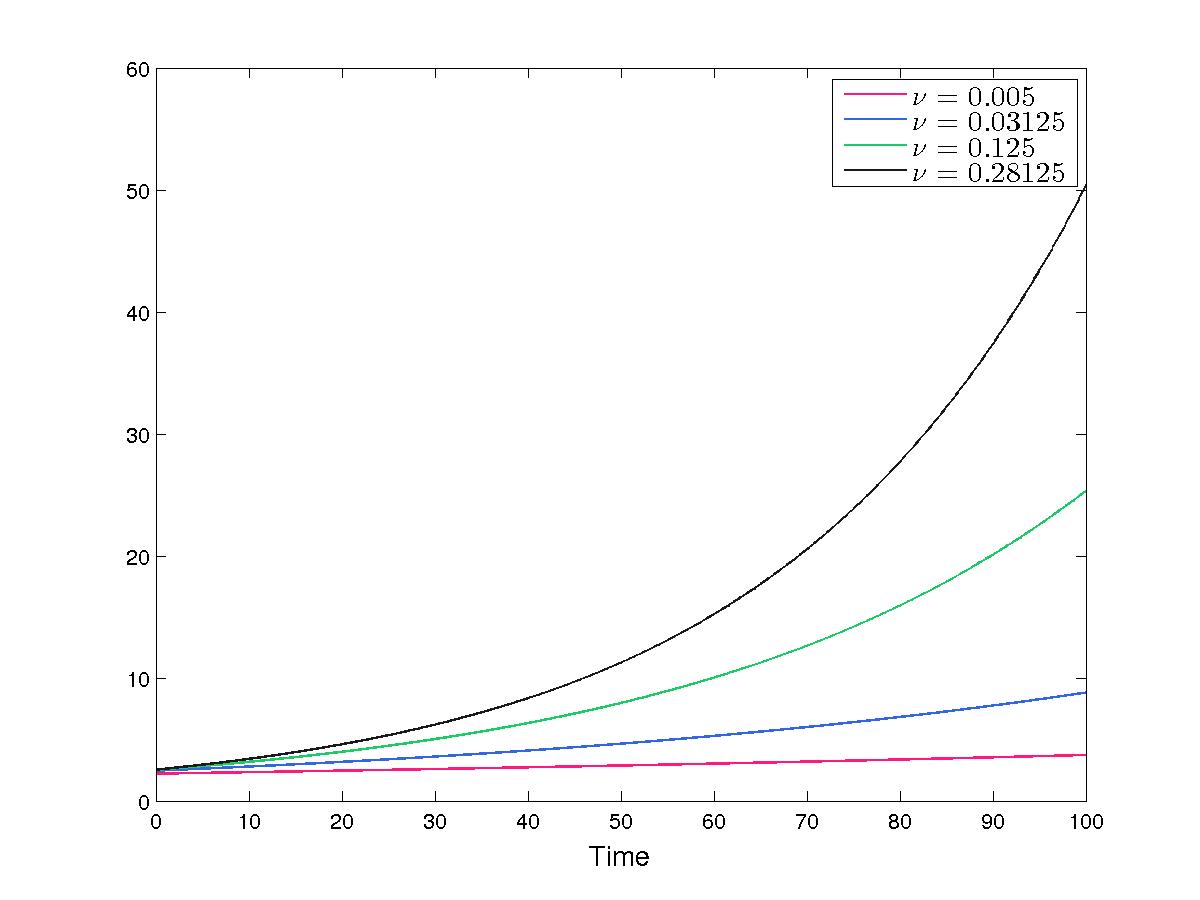}}\\
\caption{Balanced growth path solutions for different diffusivities $\nu$}\label{f:bgpdiff}
\end{figure}

\section{Conclusion}
In this paper we present a full analysis for the existence of BGP solutions of a Boltzmann mean-field game model for knowledge growth. We discuss the necessary assumptions on the
initial datum as well as the existence of degenerate BGP solutions. Furthermore we give first insights into the behavior of the model in the case of geometric
diffusion with various numerical simulations. The simulations confirm the hypothesis of Achdou et. al, that is the existence of balanced growth path solutions in the case
of diffusion but also indicate important questions which shall be addressed in the near future. The existence of degenerate BGP solutions, which already caused significant challenges in this work,
can not be excluded in the diffusive case. Therefore we will focus on the analysis of the diffusive problem as well as the construction of numerical schemes for non-degenerate BGP solutions, based on the time-dependent formulation or the rescaled problem in the near future. Finally modeling generalizations, such as more complicated interaction laws, shall be considered  as well.

\section*{Acknowledgement} 
\noindent MTW acknowledges financial support from the Austrian Academy of Sciences \"OAW via the New Frontiers Group NST-001. 
This research was funded in part by the French ANR blanche project Kibord: ANR-13-BS01-0004.
The authors thank Benjamin Moll for the helpful discussions and comments while preparing the manuscript.
\bibliographystyle{abbrv}
\bibliography{mfg_boltzmann}

\end{document}